\documentclass[12pt]{article}
\usepackage{mathrsfs}

\usepackage{t1enc}
\usepackage[latin1]{inputenc}
\usepackage[english]{babel}
\usepackage{amsmath,amsthm,amssymb}
\usepackage{amsfonts}
\usepackage{latexsym}
\usepackage{enumerate}

\usepackage{graphicx}
\usepackage[natural]{xcolor}
\usepackage{tikz}
\usepackage{tikz-3dplot}
\usetikzlibrary{shapes}
\usetikzlibrary{arrows,decorations.pathmorphing,backgrounds,positioning,fit,petri,automata}
\usetikzlibrary {positioning}

\theoremstyle{plain}
\newtheorem{thm}{Theorem}[section]
\newtheorem{lem}[thm]{Lemma}

\newtheorem{conj}[thm]{Conjecture}
\newtheorem{claim}{Claim}

\theoremstyle{plain}

\theoremstyle{plain}

\theoremstyle{plain}

\title{Hamiltonian cycles in 7-tough $(P_3\cup 2P_1)$-free graphs}
\author{Yuping Gao$^{a}$, Songling Shan$^{b}$\\
{\small a. School of Mathematics and Statistics, Lanzhou University, Lanzhou 730000, China}\\
{\small b. Department of Mathematics, Illinois State University, Normal, IL 61790, USA}}
\date{}

\begin{document}
\baselineskip 0.65cm

\maketitle
\begin{abstract} The toughness of a noncomplete graph $G$ is the maximum real number $t$ such that the ratio of $|S|$ to the number of components of $G-S$ is at least $t$ for every cutset $S$ of $G$, and the toughness of a complete graph is defined to be $\infty$.  Determining the toughness for a given graph is NP-hard. Chv\'{a}tal's toughness conjecture, stating that there exists a constant $t_0$ such that every graph with toughness at least $t_0$ is hamiltonian, is still open for general graphs. A graph is called $(P_3\cup 2P_1)$-free if it does not contain any induced subgraph  isomorphic to $P_3\cup 2P_1$, the disjoint union of $P_3$ and two isolated vertices.  In this paper, we confirm Chv\'{a}tal's toughness conjecture  for $(P_3\cup 2P_1)$-free graphs by showing that every 7-tough $(P_3\cup 2P_1)$-free graph on at least three vertices is hamiltonian.

\medskip

\end{abstract}

\section{Introduction}
All graphs considered in this paper are undirected and simple. Let $G$ be a graph. The vertex set and edge set of $G$ are denoted by $V(G)$ and $E(G)$, respectively. For a vertex $v\in V(G)$, the set of neighbors of $v$ in $G$ is denoted by $N_{G}(v)$. Let $H\subseteq G$ be a subgraph of $G$, $x\in V(G)$ and $S\subseteq V(G)$. Define $N_{H}(x)=N_{G}(x)\cap V(H)$, $N_{G}(S)=\cup_{x\in S}N_{G}(x)\setminus S$ and $N_{H}(S)=N_{G}(S)\cap V(H)$.
%Following the above notation, let $|N_{G}(v)|=d_{G}(v)$, $|N_{G}(S)|=d_{G}(S)$, $|N_{H}(x)|=d_{H}(x)$ and $|N_{H}(S)|=d_{H}(S)$. Furthermore,
We use $G[S]$ to denote the subgraph of $G$ induced by $S$ and $G-S$ to denote the subgraph $G[V(G)\setminus S]$. For simplicity, $G-\{x\}$ is written as $G-x$. If $uv\in E(G)$ is an edge, then we write $u\thicksim v$ and $u\nsim v$ otherwise. Let $V_1,V_2\subseteq V(G)$ be two disjoint vertex sets. Then the set of edges with one end in $V_1$ and the other in $V_2$ is denoted by $E_{G}(V_1,V_2)$. The subscript $G$ will be omitted in all the  notation above if no confusion may arise.

A \emph{hamiltonian path} (resp. \emph{cycle}) in a graph $G$ is a path (resp. cycle) which contains all the vertices of $V(G)$. A graph is called \emph{hamiltonian connected} if there is a hamiltonian path between any two distinct vertices. Let $c(G)$ denote the number of components in a graph $G$. Chv\'{a}tal \cite{C1973} defined a noncomplete graph $G$ to be $t$-\emph{tough} if $|S|\geq t\cdot c(G-S)$ for every subset $S\subseteq V(G)$ with $c(G-S)\geq 2$, i.e., $S$ is a \emph{cutset} of $G$. The \emph{toughness} $\tau(G)$ of $G$ is the maximum real number $t$ for which $G$ is $t$-tough or is $\infty$ if $G$ is complete. A cutset $S$ is called a \emph{tough set} in $G$ if $|S|=\tau(G)\cdot c(G-S)$.  Let $S$ be a cutset of $G$. A component of $G-S$ is said to be \emph{trivial} if it contains only one vertex, otherwise we call it \emph{nontrivial}. Since every cycle is 1-tough, every hamiltonian graph is 1-tough. However, as noted by Chv\'{a}tal, the converse only holds for graphs with at most six vertices. Figure 1 given by Chv\'{a}tal is an example of 1-tough nonhamiltonian graph with seven vertices.  Chv\'{a}tal conjectured that large enough toughness in a graph $G$ would guarantee that $G$ is hamiltonian. In \cite{C1973}, he proposed the following conjecture.

\begin{conj}\label{Conj1} {\rm(Chv\'{a}tal's toughness conjecture)} There exists a constant $t_0$ such that every $t_0$-tough graph on at least three vertices is hamiltonian.
\end{conj}

\begin{center}
	\begin{tikzpicture}[scale=0.7]
	{\tikzstyle{every node}=[draw,circle,fill=white,minimum size=4pt,
		inner sep=0pt]
		\draw (0,0) node (v1)  {}
		-- ++(0:4cm) node (v2)  {}
		-- ++(120:4cm) node (v3)  {}
		-- (v1);
		\draw (2,1.3) node (v0)  {}
		-- (v1)
		 ++(40:3.2cm) node (v4)  {}
		-- (v0);
		\draw (2.3,0.5) node (v5)  {}
		-- (v2);
        \draw (1.3,1.3) node (v6)  {}
		-- (v0)-- (v2);
		\draw (v6) -- (v1);
\draw (v5) -- (v0)-- (v3)-- (v4);}
	
	\node at (2,-0.7) {Figure 1: 1-tough nonhamiltonian graph with seven vertices};
	\end{tikzpicture}
\end{center}

In 2000, Bauer, Broersma and Veldman \cite{BBV2000} showed that for every $\epsilon>0$, there exists a $(\frac{9}{4}-\epsilon)$-tough graph with no hamiltonian path. It follows that if Chv\'{a}tal's toughness conjecture is true,  then $t_0\geq \frac{9}{4}$. Chv\'{a}tal's toughness conjecture is still open. However, it is true for a number of well-studied classes of graphs including planar graphs \cite{T1956}, interval graphs \cite{K1985}, claw-free chordal graphs \cite{BP1986}, planar chordal graphs \cite{BHT1999}, cocomparability graphs \cite{DKS1997},
split graphs \cite{KLM1996}, spider graphs \cite{KKS2007}, chordal graphs \cite{CJKL1998,KK2017}, $k$-trees ($k\geq 2$) \cite{BXY2007}, $R$-free graphs for $R\in \{P_4, K_1\cup P_3, 2K_1\cup K_2\}$ \cite{LBZ2016}, $2K_2$-free graphs \cite{BPP2014,S2020,OS2021}, $(P_2\cup P_3)$-free graphs \cite{S2021}, and $(P_2\cup 3P_1)$-free graphs \cite{HG2021}.

%\begin{itemize}
%\item $t_0>\frac{3}{2}$ for planar graphs \cite{T1956};
%\item $t_0\geq 1$ for interval graphs \cite{K1985}, claw-free chordal graphs \cite{BP1986}, planar chordal graphs \cite{BHT1999}, cocomparability graphs \cite{DKS1997} and $R$-free graphs for $R\in \{P_4, K_1\cup P_3, 2K_1\cup K_2\}$ \cite{LBZ2016};
%\item $t_0\geq \frac{3}{2}$ for split graphs \cite{KLM1996} and spider graphs \cite{KKS2007};
%\item $t_0\geq 18$ for chordal graphs \cite{CJKL1998};
%\item $t_0\geq 10$ for chordal graphs \cite{KK2017};
%\item $t_0\geq \frac{k+1}{3}$ for $k$-trees ($k\geq 2$) \cite{BXY2007};
%\item $t_0\geq 25$ for $2K_2$-free graphs \cite{BPP2014};
%\item $t_0\geq 3$ for $2K_2$-free graphs \cite{S2020};
%\item $t_0\geq 2$ for $2K_2$-free graphs \cite{OS2021};
%\item $t_0\geq 15$ for $(P_2\cup P_3)$-free graphs \cite{S2021};
%\item $t_0\geq 3$ for $(P_2\cup 3P_1)$-free graphs \cite{HG2021}.
%\end{itemize}

Let $P_{n}$ denote a path with $n$ vertices. A graph is $(P_3\cup 2P_1)$-\emph{free} if it does not contain any induced subgraph  isomorphic to $P_3\cup 2P_1$, the disjoint union of $P_3$ and two isolated vertices.
 In this paper, we study the hamiltonicity of $(P_3\cup 2P_1)$-free graphs under a given toughness condition and obtain the following result.
% Since $P_2$ is the complement of $2P_1$, the proof technique of $(P_2\cup P_3)$-free graphs in \cite{S2021} is not applicable for $(P_3\cup 2P_1)$-free graphs. The proof approach used in this paper for showing Theorem \ref{thm} below is new and may be used to seek other properties of $(P_3\cup 2P_1)$-free graphs.

\begin{thm}\label{thm} Every  $7$-tough $(P_3\cup 2P_1)$-free graph on at least three vertices  is hamiltonian.
\end{thm}

The remainder of this paper is organized as below. In next section, we introduce some
notation and preliminaries. In Section 3, we prove Theorem~\ref{thm}.

\section{Preliminaries}

We start this section with some definitions and notation. For two integers $p$
and $q$, we let $[p,q]=\{i\in \mathbb{Z}:  p\le i\le q\}$.
A \emph{star-matching} in a graph is a set of vertex-disjoint copies of stars. The vertices of degrees at least 2 in a star-matching are called the \emph{centers} of the star-matching. In particular, if every star in a star-matching is isomorphic to $K_{1,t}$, where $t\geq 1$ is an integer, we call the star-matching a $K_{1,t}$-\emph{matching}. For a star-matching $M$, we denote by $V(M)$ the set of vertices covered by $M$. And if $x,y\in V(M)$ and $xy\in E(M)$, we say $x$ is a \emph{partner} of $y$.

Let $C$ be an oriented cycle. For $x\in V(C)$, denote the immediate successor of $x$ on $C$ by $x^{+}$. For $u,v\in V(C)$, $u\overrightarrow{C}v$ denotes the segment of $C$ starting with $u$, following $C$ in the orientation, and ending at $v$. Likewise, $u\overleftarrow{C}v$ is the opposite segment of $C$ with ends $u$ and $v$. We assume all cycles in consideration afterwards are oriented. Let $P$ be a path, $u,v\in V(P)$ be two vertices in $P$, the segment of $P$ with ends $u$ and $v$ is denoted by $uPv$. Let $uPv$ and $xQy$ be two disjoint paths. If $v$ is adjacent to $x$, we write $uPvxQy$ as the concatenation of $P$ and $Q$ through the edge $vx$.

We will need the following Lemmas  in our proof.

\begin{lem}\label{2.1}{\rm(\cite{AK2011}, Theorem 2.10)} Let $G$ be a bipartite graph with partite sets $X$ and $Y$, and $f$ be a function from $X$ to the set of positive integers. If for every $S\subseteq X$, $|N_G(S)|\geq \sum\limits_{v\in S}f(v)$, then $G$ has a subgraph $H$ such that $X\subseteq V(H)$, $d_H(v)=f(v)$ for every $v\in X$, and $d_{H}(u)=1$ for every $u\in Y\cap V(H)$.
\end{lem}

\begin{lem}\label{2.2}{\rm(\cite{BBVV1995})} Let $t>0$ be a real number and $G$ be a $t$-tough graph on $n\ge 3$ vertices  with $\delta(G)>\frac{n}{t+1}-1$. Then $G$ is hamiltonian.
\end{lem}

\begin{lem}\label{lem:degree-sum}{\rm(\cite{SS2021})} Let $t>0$ be a real number and $G$ be a $t$-tough graph on $n\geq 3$ vertices. If the degree sum of any two nonadjacent vertices of $G$ is greater than $\frac{2n}{t+1}+t-2$, then $G$ is hamiltonian.
\end{lem}

\begin{lem}\label{2.3}{\rm(\cite{SS2021})} Let $t \ge 1 $ be a real number, $G$ be a $t$-tough graph on $n\ge 3$ vertices and $C$ be a nonhamiltonian cycle of $G$. If $x\in V(G)\setminus V(C)$ satisfies that $d_{C}(x)>\frac{n}{t+1}-1$, then $G$ has a cycle $C'$ such that $V(C')=V(C)\cup \{x\}$.
\end{lem}

%\begin{lem}\label{lem:path-insertion}
%	Let $G$ be a $t$-tough  noncomplete graph for some $t>0$ on $n\ge 3$ vertices and $C$ be a  nonhamiltonian cycle of $G$. For an $xy$-path  $P$ in $G-V(C)$, if $d_C(x), d_C(y) >\frac{n}{t+1}-1$ and $|N_C(x)\cap N_C(y)|>\frac{n}{t+1}-1$, then $G$ has a cycle $C'$
%	such that $V(C')=V(C)\cup V(P)$.
%\end{lem}
%\proof Since $G$ is noncomplete, $\delta(G) \le n-2$ and so $\tau(G) \le \frac{n-2}{2}$.
%Thus $\frac{n}{t+1}\ge 2$.
%Let $W=N_C(x)\cap N_C(y)$. It is clear that for any two vertices $u,v\in W$, $uv\not\in E(C)$. Define $W^+=\{w^+: w\in W\}$ to be the set of successors of vertices of $W$ on $C$.
%We claim that $W^+$ is an independent set in $G$. For otherwise, there exist $u,v\in W$
%such that $u^+v^+\in E(G)$. Assume $u, u^+, v$ appear in the order $u, u^+, v $ along
%$\overrightarrow{C}$. Then  $C'=uxPyv\overleftarrow{C}u^+v^+\overrightarrow{C}u$ is a desired cycle. Therefore $W^+$ is an independent set in $G$. Since any two vertices $u,v\in W$ satisfying  $uv\not\in E(C)$, $W^+\cup \{x\}$ is also an independent set of $G$ with $|W^+\cup \{x\}|>\frac{n}{t+1}-1+1=\frac{n}{t+1}$. Let $S=V(G)\setminus (W^+\cup \{x\})$. We have $c(G-S)=|W^+\cup \{x\}|>\frac{n}{t+1} \ge 2$, implying $\tau(G)<t$, giving a contradiction.
%\qed
%\begin{lem}\label{2.6}{\rm(\cite{D1952,O1960})} Let $G$ be a connected graph with $n$ vertices and $\delta(G)\geq \frac{n+1}{2}$. Then $G$ is hamiltonian connected.
%\end{lem}

\begin{lem}\label{2.4} Let $G$ be a  more than 1-tough  $(P_3\cup P_1)$-free graph on $n\ge 3$ vertices. Then $G$ is hamiltonian connected.
\end{lem}

\begin{proof}
Let $u,v\in V(G)$ be any two distinct vertices. We find in $G$
a longest $uv$-path $P$. We may assume  that $P$	 is not a hamiltonian path of
$G$. Let $D$ be a component of $G-V(P)$, $W=N_P(V(D))=\{x_0, x_1, \ldots, x_\ell\}$. Then $\ell\ge 2$ by the toughness of $G$. Assume
for each $i\in [1,\ell]$, $x_{i-1}$ is in between $u$ and $x_i$ on $P$,
where note $x_0$ could be the same as $u$ and $x_\ell$ could be the same as $v$.
Let $h=\ell$ if $x_\ell \ne v$, and $h=\ell-1$
if $x_\ell=v$, and
let $W^+=\{x_i^+: i\in [0,h]\}$
be the set of the immediate successors of vertices  $x_i$ on $P$
along the direction from $u$ to $v$.

For any $i\in [0,\ell-1]$, we claim that  $x_ix_{i+1}\not\in E(P)$. For otherwise, let $x_i', x_{i+1}'\in V(D)$ such that $x_{i}'x_i,x_{i+1}'x_{i+1}\in E(G)$, and let $P'$
be an $x_{i}'x_{i+1}'$-path in $D$. Then $uPx_ix_i'P'x_{i+1}'x_{i+1}Pv$ is a longer $uv$-path than $P$.

For any two distinct  $i,j\in [0,h]$,  we claim that $x_i^+x_j^+ \not\in E(G)$.
For otherwise, let $x_i', x_{j}'\in V(D)$ such that $x_{i}'x_i,x_{j}'x_{j}\in E(G)$, and let $P'$
be an $x_{i}'x_{j}'$-path in $D$. Assume, by symmetry that  $i<j$.
Then $uPx_ix_i'P'x_{j}'x_jPx_i^+x_j^+Pv$ is a longer $uv$-path than $P$.

For each $i\in [0,\ell-1]$, let $L_i=x_iPx_{i+1}-\{x_i,x_{i+1}\}$.
Since $x_ix_{i+1} \not\in E(P)$, each $L_i$ is a path with at least one vertex.
Since $G$ is $(P_3\cup P_1)$-free, each component of $G-W-V(D)$ is a complete graph.
For any two distinct $i,j\in [0,h]$, as $x_i^+x_j^+ \not\in E(G)$, $L_i$
and $L_j$ are contained in distinct components of $G-W$. Those components containing $L_i$'s together with $D$ gives $c(G-W) \ge \ell +1 \ge 3$.
Thus $\tau(G) \le \frac{|W|}{c(G-W)}  \le 1$, giving a contradiction to $\tau(G)>1$.

\end{proof}

Let $G$ be a graph, $S$ be a cutset of $G$, and $D_1,D_2,\ldots,D_{\ell}$
be all the components of $G-S$. For any $i, j\in [1,\ell]$, if there exists $S_i \subseteq N_G(V(D_i))\cap S$  such that (i) $|S_i|=2s$ for some integer $s\ge 1$, (ii)  if $|V(D_i)| \ge 2$, then $S_i$ can be partitioned into $S_{i1}$ and $S_{i2}$ and $N_{D_i}(S_i)$ can be  partitioned into $W_{i1}$ and $W_{i2}$  with $|S_{i1}|=|S_{i2}|=s$ such that $S_{i1}\subseteq N_G(W_{i1})\cap S$ and $S_{i2}\subseteq N_G(W_{i2})\cap S$; and
 (iii) $S_i\cap S_j=\emptyset$ if $i\ne j$, then we say $G$ has a \emph{generalized  $K_{1,2s}$-matching}
with centers as  components of $G-S$, and call vertices in $S_i$ the \emph{partners}
of $D_i$ from $S$.

\begin{lem}\label{2.5} Let $G$ be a $t$-tough graph on $n$ vertices  for some $t\ge 2$,  $S$ be a cutset in $G$, and let $s=\lfloor t/2\rfloor $. Then $G$
has a generalized  $K_{1,2s}$-matching with centers as  components of $G-S$.

\end{lem}

\begin{proof} Let $D_1,D_2,\ldots,D_{\ell}$
	be all the components of $G-S$.
	For each $D_i$, as $\tau(G)\geq 2$, $| N_G(V(D_i)) \cap S|\geq 2t$.
		 Furthermore, if $|V(D_i)|\geq 2$, then $|N_{D_i}(S)|\geq 2$. Otherwise, if $|N_{D_i}(S)|=1$, then the neighbor of $S$ in $D_i$ would be a cutvertex in $G$.

We will construct a bipartite graph $H$ in the following steps. For each $D_i$ with $|V(D_i)|\geq 2$, let $N_{D_i}(S)=W_i$. Let $W_i^{1}\cup W_i^{2}$ be a partition of $W_i$ such that both of them are nonempty. We contract all vertices in $W_{i}^1$ into a single vertex $u_i$ and all vertices in $W_{i}^2$ into a single vertex $v_i$. Then $G[D_i]$ is changed into an edge $u_iv_i$. The edges between $D_i$ and $S$ are now between $u_iv_i$ and $S$. For each $D_i$ with $|V(D_i)|=1$, let $V(D_i)=\{w_i\}$. Split $w_i$ into two vertices $u_i$ and $v_i$, and distribute the edges in $G$ incident with $w_i$ between $u_i$ and $v_i$.

Let $T=\{u_i,v_i:1\leq i\leq \ell\}$ be the collection of vertices corresponding to $D_i's$ and let $H=H[S,T]$ be a bipartite graph with vertex set $V(H)=S\cup T$ and edge set $E(H)=E(S,\cup_{i=1}^{\ell}V(D_i))$. To complete the proof of Lemma \ref{2.5}, it is sufficient to prove that $H$ has a $K_{1,2s}$-matching saturating $T$. Suppose not, by Lemma~\ref{2.1}, there exists a nonempty subset $T^*\subseteq T$ such that
\begin{equation}\label{eq1}
|N_{H}(T^*)|<s|T^*|.
\end{equation}
As $T^{*}$ corresponds to at  least $\frac{|T^{*}|}{2}$ components of $G-S$, \eqref{eq1} implies that
\begin{equation*}
c(G-N_{H}(T^{*}))\geq \frac{|T^{*}|}{2}=\frac{s|T^{*}|}{2s}>\frac{1}{2s}|N_{H}(T^{*})|,
\end{equation*}
giving that $\tau(G)<2s\le t$, a contradiction to the toughness of $G$. This completes the proof of Lemma \ref{2.5}.
\end{proof}

\section{Proof of Theorem \ref{thm}}
\proof  We may assume $G$ is not a complete graph.
Since $G$ is 7-tough and noncomplete, $G$ is 14-connected and $\delta(G)\geq 14$. By Lemma \ref{2.2}, we may further assume  $\delta(G)\leq \frac{n}{8}-1$. It follows that
\begin{equation}\label{eqn1}
n\geq 8\delta(G)+8\geq 120.
\end{equation}
Furthermore, by Lemma~\ref{lem:degree-sum}, we may assume that
there exist two nonadjacent vertices in $G$ with degree sum at most  $ \frac{n}{4}+5$.
Let $u,v\in V(G)$ be two nonadjacent vertices such that $d(u)+d(v)\leq \frac{n}{4}+5$. Let $$S_{uv}=N(u)\cup N(v).$$ Then $|V(G)\setminus (S_{uv}\cup \{u,v\})|\geq n-\frac{n}{4}-5-2\geq 1$. It follows that $c(G-S_{uv})\geq 3$. The following claim is obvious by the $(P_3\cup 2P_1)$-freeness of $G$.

\begin{claim}\label{claim1}Each component of $G-S_{uv}$ is a complete graph.
\end{claim}

Let
\begin{eqnarray*}
S_u&=&\{x\in S_{uv}: N_G(x)\cap (V(G)\setminus S_{uv})=\{u\}\}, \\
S_v&=&\{x\in S_{uv}\setminus S_u: N_G(x)\cap \left(V(G)\setminus (S_{uv} \setminus S_u)  \right)=\{v\}\}, \\
S&=& S_{uv}\setminus (S_u\cup S_{v}).
\end{eqnarray*}
%Let $S_{u}$ be the set of vertices in $S_{uv}$ which are only adjacent to $u$ in $G-S_{uv}$, $S_{v}$ be the set of vertices in $S_{uv}\setminus S_{u}$ which are only adjacent to $v$ in $G-(S_{uv}\setminus S_{u})$, and $S=S_{uv}\setminus (S_u\cup S_{v})$.
In such a construction, $u$ and $v$ belong to distinct components of $G-S$. Then we have the following claim by the construction of $S$ and $(P_3\cup 2P_1)$-freeness of $G$.

\begin{claim}\label{claim2}\begin{enumerate}[{\rm(i)}]

%\item $G[S_{u}]$ and $G[S_{v}]$ are both complete graphs;

\item $c(G-S)=c(G-S_{uv})\geq 3$ and each component of $G-S$ is a complete graph;

\item for any vertex $x\in S$, $x$ is adjacent to vertices from at least two components of $G-S$;

\item if $c(G-S)\geq 4$, then for each vertex $x\in S$, $x$ is not adjacent to all vertices of at most one component of $G-S$.
\end{enumerate}
\end{claim}

{\bf \noindent Case 1} For any component $D$ of $G-S_{uv}$, $|V(G)\setminus (S_{uv}\cup V(D))|> \frac{n}{8}-1$.

Since $G-S_{uv}$ has at least one component $D^*$ which does not contain $u$ or $v$, by the assumption of Case 1, we know that $G-S_{uv}$ has at least one component other than $D^*$  not containing $u$ or $v$. It follows that $c(G-S_{uv})\geq 4$.  Then $c(G-S)\geq 4$. Furthermore, by Claim \ref{claim2}(iii), the following claim holds.

\begin{claim}\label{claim3}For each vertex $x\in S$, we have  $d_{G-S}(x)> \frac{n}{8}-1$.
\end{claim}

 In the following, we will construct a cycle that covers all the vertices of $G-S$ firstly, then insert the remaining vertices of $S$ into the cycle by applying Lemma~\ref{2.3} repeatedly.
Since $G$ is 7-tough and so is 2-tough, applying Lemma~\ref{2.5} with $t=2$, there exist distinct vertices $x_1,y_1,x_2,y_2,\ldots,x_t,y_t$ in $S$  such that $x_i,y_i$ are partners of $D_i$. As each $D_i$ is a complete graph, there exists a hamiltonian path $P_i$ in $D_i$ such that the two ends $a_i,b_i$ of $P_i$ are adjacent to $x_i$ and $y_i$, respectively. Note that $a_i=b_i$ if $|V(D_i)|=1$. Since $y_1$ is not adjacent to at most one component of $G-S$ by Claim \ref{claim2}(iii), without loss of generality, we may assume that $y_1\thicksim a_2$. Recursively we can assume that $y_i\thicksim a_{i+1}$ for each $i\in[1,t-2]$. If both of $y_{t-1}\sim a_t$ and $x_1\sim b_t$ hold, then we get a cycle $C=x_1a_1P_1b_1y_1a_2P_2b_2y_2\ldots y_{t-1}a_{t}P_tb_{t}x_1$ that contains all vertices of $\cup_{i=1}^{t}V(D_i)$.

Thus, $y_{t-1}\nsim a_{t}$ or $x_1\nsim b_t$. By Claim \ref{claim2}(iii), we have that $y_{t-1}\thicksim a_1$ or $x_1\thicksim b_{t-1}$. In either case, we can find a cycle $C'$ such that $\cup_{i=1}^{t-1}V(D_i)\subseteq V(C')$. We will extend $C'$ to a larger cycle $C$ containing also vertices of $D_t$ in the following and give a claim first.

\begin{claim}\label{claim4} \emph{(i)} For any two vertices $u\in N_{C'}(x_t)$ and $w\in N_{C'}(y_t)$, $uw\not\in E(C')$.

\emph{(ii)}   $x_t$ and $y_t$ can not be adjacent to all vertices of any nontrivial component of $G-S$.  As a consequence, $|V(D_i)|=1$ for each $i\in [3,t]$.
\end{claim}

\begin{proof} (i) Otherwise $C=(C'\setminus \{uw\})\cup \{ux_t,y_tw\}\cup P_t$ is a desired cycle containing all vertices of $G-S$.

(ii) Let $D$ be a nontrivial component of $G-S$ such that both of $x_t$ and $y_t$ are adjacent to all vertices of $D$. By the construction of $C'$ and $D$ being a complete graph, we know that there exist two adjacent vertices $u,w\in V(C')\cap V(D)$ such that $uw\in E(C')$, contradicting (i). So for each component $D$ of $G-S$ such that $x_t$ and $y_t$ are both adjacent to all vertices of $D$, it holds that $|V(D)|=1$. By Claim \ref{claim2}(iii), for any vertex $x\in S$, each of $x_t$ and $y_t$ is not adjacent to all vertices of at most one component of $G-S$. Since $|V(D_1)|\geq |V(D_2)|\geq \ldots\geq |V(D_t)|$, it follows that $|V(D_i)|=1$ for each $i\in[3,t]$.
\end{proof}

As $G$ is 7-tough and $|S| \le \frac{n}{4}+5$, it follows that
$$|V(D_1)|+|V(D_2)| \ge n-|S|-\left(\frac{|S|}{7}-2\right)>\frac{5n}{7}-4>81$$
by~\eqref{eqn1}. By the assumption of Case 1, we know that
$$
|V(D_2)|>\frac{n}{8}-1-\left(\frac{|S|}{7}-2\right)>\frac{5n}{56}>10
$$
also by~\eqref{eqn1}.

Thus both $D_1$ and $D_2$ are nontrivial components of $G-S$.
By Claim \ref{claim4}(ii), we assume, without loss of generality, that $x_t$ is adjacent to all vertices of $D_1$ and $y_t$ is adjacent to all vertices of $D_2$. If $y_{t-1}\thicksim a_1$, then $$C=x_tb_1P_1a_1y_{t-1}a_{t-1}y_{t-2}a_{t-2}\ldots y_3a_3y_2b_2P_2a_2y_ta_tx_t$$ is a cycle containing all vertices of $\cup_{i=1}^{t}V(D_i)$. If $x_1\thicksim b_{t-1}$, then $$C=x_tb_1P_1a_1x_1a_{t-1}y_{t-2}a_{t-2}y_{t-3}\ldots a_3y_2b_2P_2a_2y_ta_tx_t$$ is a cycle containing all vertices of $\cup_{i=1}^{t}V(D_i)$. Note that $a_i=b_i$ for each $i\in [3,t]$ by Claim~\ref{claim4}(ii). By Claim \ref{claim3}, for each vertex $x\in S\setminus V(C)$, $d_{C}(x)>\frac{n}{8}-1$. Applying  Lemma~\ref{2.3} recursively  on vertices of $S$
that are
out of the current cycle  containing $V(C)$,
   we get a hamiltonian cycle of $G$.

{\bf  \noindent Case 2} There exists a component $D_0$ of $G-S_{uv}$ such that $|V(G)\setminus (S_{uv}\cup V(D_0))|\leq\frac{n}{8}-1$.

Let the components of $G-S$ be $D_1,D_2,\ldots,D_t$ with $|V(D_1)|\geq |V(D_2)|\geq \ldots \geq |V(D_t)|$ for some integer $t\ge 3$. In this case, $|V(D_1)| \ge  |V(D_0)|\geq\frac{5n}{8}-4$ since $|S_{uv}|\leq \frac{n}{4}+5$. Let $S_1\subseteq S$ be a largest subset of $S$ such that $|N_{D_1}(S_1)|< 2|S_1|$, $S_2=S\setminus S_1$ and $S^{*}=N_{D_1}(S_1)\cup S_2$.
By the choice of $S_1$, it can be seen that for any subset $T\subseteq S_2$, $|N_{D_1}(T)\setminus N_{D_1}(S_1)|\geq 2|T|$. By Lemma \ref{2.1}, there exists a $K_{1,2}$-matching $M$ between $S_2$ and $V(D_1)\setminus N_{D_1}(S_1)$ with centers as vertices of $S_2$. Let $Q=V(M)\setminus S_2$ and  $D_1^{*}=D_1-N_{D_1}(S_1)-Q$.

As $|V(D_1)| \ge \frac{5n}{8}-4$, $|S_{uv}|\leq \frac{n}{4}+5$, and $|N_{D_1}(S_1)|< 2|S_1|$, it follows that
\begin{equation}\label{eqn2}
|V(D_1^*)| >\frac{5n}{8}-4-2\left(\frac{n}{4}+5\right)=\frac{n}{8}-14 \ge 1,
\end{equation}
where the last inequality above follows by~\eqref{eqn1}.  Thus $G-S^*$ has a component containing $D_1^*$.

\begin{claim}\label{claim5} If there are disjoint paths $Q_1, \ldots, Q_k$
	for some integer $k\ge 1$ in $G$ such that (i) vertices of $G-S^*-V(D_1)$
	are contained as internal vertices of one or more $Q_i$'s, and (ii) each $Q_i$
	has its two endvertices from $S^*$ and those two are the only vertices of $S^*$
	that are contained in $Q_i$, then $G$ has a hamiltonian cycle.
\end{claim}

\proof Let $i\in [1,k]$ and  $x_i,y_i$ be the endvertices of $Q_i$, where note
$x_i,y_i \in S^*$.
If $x_i,y_i\in N_{D_1}(S_1)$, then let $P_i=Q_i$. If $x_i\in N_{D_1}(S_1),y_i\in S_2$, then let $y_i'$ be a partner of $y_i$ and $P_i=x_iQ_iy_iy_i'$. If $x_i\in S_2,y_i\in N_{D_1}(S_1)$, then let $x_i'$ be a partner of $x_i$ and $P_i=x_i'x_iQ_iy_i$. If $x_i,y_i\in S_2$, then let $x_i'$ and $y_i'$ be a partner of $x_i$ and $y_i$, respectively, and  $P_i=x_i'x_iQ_iy_iy_i'$. Then the endvertices of $P_i$ belong to $V(D_1)$ and all those paths are pairwise disjoint.
 Furthermore, for each $x\in S_2\setminus (\cup_{i=1}^{k}V(P_i))$, the partners of $x$  belong to $V(D_1)$ and are not contained in any $P_i$.
Since $D_1$ is a complete graph, there exists a hamiltonian cycle $C$ in $D_1$ such that the two endvertices  of each $P_i$ are consecutive on $C$ and the two partners of each $x\in S_2\setminus (\cup_{i=1}^{k}V(P_i))$ are also consecutive on $C$.
Now for each edge $wz\in E(C)$, if $w$ and $z$ are endvertices of some $P_i$,  we replace $wz$ by $P_i$; if $w$ and $z$ are the two partners of
some  $x\in S_2\setminus (\cup_{i=1}^{k}V(P_i))$, we replace $wz$ by $wxz$.
After doing this replacement for all such edges $wz$ of $C$, we have obtained a hamiltonian cycle of $G$.
\qed

{\bf \noindent  Subcase 2.1} $c(G-S^*)\geq3$.

In this case, each component of $G-S^{*}$ is a complete graph by the freeness of $P_3\cup 2P_1$. Let $R_1,R_2,\ldots,R_{\ell}$ be all the components of $G-S^{*}$ with $R_i\neq D_1-N_{D_1}(S_1)$ for each $i\in [1,\ell]$. Applying Lemma \ref{2.5} with $t=2$,
$G$ has a generalized $K_{1,2}$-matching with centers as $R_1,R_2,\ldots,R_{\ell}$.
We let $x_i$ and $y_i$ be the partners of $R_i$ from $S^{*}$, and let
 $a_i,b_i \in V(R_i)$ such that $a_ix_i, b_iy_i\in E(G)$, where  note that $a_i=b_i$ if $|V(R_i)|=1$. Moreover, there is a hamiltonian path $P_i'$ from $a_i$ to $b_i$ in $R_i$.
 Then we are done by Claim~\ref{claim5} by letting $Q_i=x_ia_iP_i'b_iy_i$ for each $i\in[1,\ell]$.

{\bf \noindent  Subcase 2.2} $c(G-S^*)=2$.

Let  $D_2^{*}$ be the other component of $G-S^*$ other than the component containing $D_1^*$. As $c(G-S) \ge 3$ and $D_1^*$ is a subgraph of the original component $D_1$ of $G-S$, it follows that $|V(D_{2}^*)|\geq 2$.
Note that  $D_2^{*}$ is $(P_3\cup P_1)$-free by the $(P_3\cup 2P_1)$-freeness of $G$. Since $G$ is 14-connected and so is 2-connected, there exist two distinct vertices $a_0,b_0\in V(D_{2}^*)$ and distinct vertices $x_0,y_0\in S^*$
such that $a_0x_0,b_0y_0\in E(G)$. If $D_2^*$ is hamiltonian connected, let $P$ be a hamiltonian path in $D_2^*$ between $a_0$ and $b_0$.  Define  $Q_1=x_0a_0Pb_0y_0$, then we apply Claim~\ref{claim5} to find a hamiltonian cycle of $G$.

Thus assume that $D_2^{*}$ is not hamiltonian connected.   Applying Lemma \ref{2.4}, we conclude  that $\tau(D_2^{*})\leq1$.
Let $W$ be a minimum tough set of $D_2^{*}$. Each component of $D_2^{*}-W$ is a complete graph as $D_2^{*}$ is $(P_3\cup P_1)$-free. Moreover, each vertex $x\in W$ is adjacent to at least two components of $D_2^{*}-W$ by $W$ being a tough set of $D_2^*$.

{\bf \noindent  Subcase 2.2.1} $c(D_2^{*}-W)=2$.

Then $1\leq |W|\leq 2$. Let $F_1$ and $F_2$ be the two components of $D_2^{*}-W$. Since $G$ is 7-tough, $G_1=G-W$ is 6-tough and so is 2-tough.
We can find in $G_1$ a generalized $K_{1,2}$-matching with centers as $F_1$ and $F_2$.
If $|W|=1$ or $W$ is a clique, then as each $F_i$ is a complete graph and each $F_i$ has two partners from $S^*$ such that these partners have at least two neighbors from $F_i$ in $G$ if $|V(F_i)| \ge 2$, we can find a hamiltonian
$ab$-path $P$ of $D_2^*$ such that $a\sim x$ and $b\sim y$ for distinct $x,y\in S^*$.
Now letting $Q_1=xaPby$, we can find a hamiltonian cycle of $G$ by Claim~\ref{claim5}.

Thus we assume $W=\{w_1,w_2\}$ and $w_1\not\sim w_2$. Let $x_i,y_i$ be two partners of $F_i$ from $S$ and $a_i,b_i\in V(F_i)$
such that $a_ix_i, b_iy_i\in E(G)$ for each $i\in [1,2]$.
Assume first that $|V(D_2^*)| \le 7$.  As $G$ is 7-tough and noncomplete,  $\delta(G) \ge 14$. Thus we can find distinct vertices $x_3,y_3, x_4,y_4\in S\setminus\{x_1,x_2,y_1,y_2\}$ such that $w_1\sim x_3, y_3$ and $w_2 \sim x_4,y_4$.
We find in $F_1$ a hamiltonian  $a_1b_1$-path $P'_1$, in $F_2$
a hamiltonian  $a_2b_2$-path $P'_2$. Let $Q_1=x_1a_1P_1'b_1y_1$, $Q_2=x_2a_2P_2'b_2y_2$,
$Q_3=x_3w_1y_3$ and $Q_4=x_4w_2y_4$. Then we apply Claim~\ref{claim5} to find a hamiltonian cycle of $G$.

Thus we have $|V(D_2^*)| \ge 8$. As $G$ is 7-tough and noncomplete, it is 14-connected.
Note also that $|S^*| \ge 2\tau(G) \ge 14$.  By the connectivity, there are distinct $a_i\in V(D_2^*)$ and distinct $x_i\in S^*$ such that $a_i\sim x_i$ for each $i\in [1,7]$.
As $W$ is a tough set of $D_2^*$,
we conclude that $\tau(D_2^*)=1$ and so $D_2^*$ is 2-connected.
Since each of $F_1$ and $F_2$ is a complete graph and $|W|=2$, we can find a hamiltonian cycle $C$
of $D_2^*$  such that $a_ia_j\in E(C)$ for some distinct $i,j\in [1,7]$.
Now letting $Q_1$ be obtained from $C$ by deleting $a_ia_j$ and adding $a_ix_i$ and $a_jx_j$,
we can apply Claim~\ref{claim5} to find a hamiltonian cycle of $G$.

{\bf \noindent  Subcase 2.2.2} $c(D_2^{*}-W)\geq 3$.

Then $D_2^*$ contains a complete bipartite graph with bipartition as $W$ and $V(D_2^{*})\setminus W$ by $(P_3\cup P_1)$-freeness of $D_{2}^{*}$.  Since $G$ is 7-tough and so $G_1=G-\{x_0,y_0\}$
is 6-tough,  applying Lemma \ref{2.5} with $t=4$, $G_1$ has a generalized $K_{1,4}$-matching
with centers as components of $D_2^*-W$.
If a component  $D$ of  $D_2^*-W$ has at least two vertices, then $V(D)$
has two disjoint subsets such that vertices from each of them are adjacent in $G$
to two distinct vertices from $S^*\cup W$.
Thus
for at least $c(D_2^{*}-W)-\frac{1}{2}|W|$ components $F_i$ of $D_2^{*}-W$, we can find a
generalized $K_{1,2}$-matching in $G_1-W$ with them as centers.

 Let $\mathcal{F}$ be the collection of all components of $D_2^{*}-W$ and $\mathcal{F}^{*}$ be those that are centers of the generalized $K_{1,2}$-matching of $G_1-W$. If $a_0,b_0\in W$, let $\mathcal{F}_1\subseteq \mathcal{F}^{*}$ such that $|\mathcal{F}_1|=|\mathcal{F}|-(|W|-1)$.
 If $a_0\in W$ and $b_0\in V(D_2^*-W)$, then we let $\mathcal{F}_1\subseteq \mathcal{F}^{*}$ such that $|\mathcal{F}_1|=|\mathcal{F}|-|W|$ and that the vertex $b_0$ is not contained in any component of $\mathcal{F}_1$ (requires at most $|\mathcal{F}|-|W|+1$ components from $\mathcal{F}^{*}$). If  $a_0, b_0\in V(D_2^*-W)$, then we let $\mathcal{F}_1\subseteq \mathcal{F}^{*}$ such that $|\mathcal{F}_1|=|\mathcal{F}|-(|W|+1)$
 and none of the vertices $a_0$ and $b_0$ is  contained in any component of $\mathcal{F}_1$ (requires at most $|\mathcal{F}|-(|W|+1)+2$ components from $\mathcal{F}^{*}$).
 Note that such $\mathcal{F}_1$ exists as  at least  $c(D_2^{*}-W)-\frac{1}{2}|W|$ components  of $D_2^{*}-W$ have two partners from $S\setminus \{x_0,y_0\}$, and
 $$
 c(D_2^{*}-W)-\left \lfloor\frac{1}{2}|W| \right\rfloor \ge |\mathcal{F}|-(|W|-1) =c(D_2^{*}-W)-|W|+1.
 $$
  Let $\mathcal{F}_2=\mathcal{F}\setminus \mathcal{F}_1$.
   Since the induced subgraph between $W$ and $D_2^{*}-W$ is a complete bipartite graph, there is a hamiltonian path  $P$ between $a_0$ and $b_0$ containing all vertices from $W$ and components in $\mathcal{F}_2$.  For each $F_i\in \mathcal{F}_1$, let $x_i,y_i$ be its partners from $S^*$ and $a_i,b_i\in V(F_i)$ such that $a_ix_i, b_iy_i\in E(G)$.
   As each component of $D_2^{*}-W$ is complete, there is a path $P'_i$ between $a_i$ and $b_i$ containing all vertices of each $F_i\in \mathcal{F}_1$.
   Now let $Q_0=x_0a_0Pb_0y_0$, and $Q_i=x_ia_iP_i'b_iy_i$ for each $F_i\in \mathcal{F}_1$.
   Applying Claim~\ref{claim5}, we find a hamiltonian cycle of $G$.
    \qed

%\section{Conclusion and open problems}

\bibliography{toughness}

\end{document}